\documentclass[oneside,reqno,12pt]{amsart}
\usepackage[greek,english]{babel}
\usepackage{amsthm}
\usepackage{amsbsy}
\usepackage{amsfonts}
\usepackage{graphicx}
\usepackage{hyperref}
\hypersetup{colorlinks=true, citecolor=blue, linkcolor=red}

\newtheorem{thm}{Theorem}
\newtheorem{cor}{Corollary}
\newtheorem{lem}{Lemma}

\newtheorem{rem}{Remark}

\begin{document}
\title[Liouville property for  a  semilinear  heat equation]{A Liouville property for eternal solutions to a  supercritical semilinear heat equation}
\author{Christos Sourdis}
\address{University of Athens, Greece.
}
\email{sourdis@uoc.gr}

\date{\today}
\begin{abstract}
We are concerned with  solutions to the nonlinear heat equation
$u_t=\Delta u+|u|^{p-1}u$, $x\in \mathbb{R}^N$, that are defined for all positive and negative time.
If the exponent $p$ is greater or equal to the Joseph-Lundgren exponent $p_c$
and $|u|$ stays below some positive radially symmetric  steady state, under a mild   condition on the behaviour of $u$ as $|x|\to \infty$,  we  show that $u$ is independent of time. Our method of proof uses Serrin's sweeping principle, based on the strong maximum principle, applied to the linearized equation for $u_t$.
 Our result covers that of   Pol\'{a}\v{c}ik and Yanagida [JDE (2005)] who had further assumed that the solution stays above some positive radial steady state and $p>p_c$.   In contrast, they relied on the use of similarity variables and invariant manifold ideas. Remarkably, to the best of our knowledge,  a corresponding Liouville property was previously missing for $p=p_c$. We emphasize that such Liouville type theorems imply the quasiconvergence of a class of solutions to the corresponding Cauchy problem.
As our viewpoint originates from the study of elliptic problems, we can prove new rigidity results for the corresponding steady state problem that are motivated
by the aforementioned ones for the parabolic flow. \end{abstract}

 \maketitle

\section{Introduction}\subsection{Motivation and known results}\label{subsecMotiv8}\subsubsection{The Cauchy problem for $p>1$}\label{subsub1}
A lot of studies have been devoted to the asymptotic behaviour as $t\to +\infty$ of solutions to the Cauchy problem:
\begin{equation}\label{eqCauchy}
\left\{\begin{array}{l}
        u_t=\Delta u+|u|^{p-1}u,\ \ x\in\mathbb{R}^N,\ t>0, \\
        u(x,0)=u_0(x), \ x\in\mathbb{R}^N,
      \end{array}
\right.
\end{equation}
with $N\geq 1$, $p>1$ and $u_0$ continuous such that $u_0(x)\to 0$ as $|x|\to \infty$.
We refer the interested reader to the monograph \cite{qs}  for an up to date account.
Despite of its simple appearance, this problem  is quite challenging and provides a rich source  of interesting phenomena. Moreover, its studies   can give insights to more complicated problems in mathematical physics.

We  point out that (\ref{eqCauchy}) is (locally) well posed in $L^\infty(\mathbb{R}^N)$ as well as in $L^q(\mathbb{R}^N)$ with $q\geq 1$ and $q>N(p-1)/2$
(see \cite[Ch. II]{qs});
a detailed study of the regularity properties of the solutions can be found in the monographs \cite{lad} and \cite{L}. Of particular interest are nonnegative initial functions $u_0$, which as is well known become immediately positive by the strong maximum principle (see for instance \cite[Subsec. 6.4.2]{Evans}),
but we will not restrict our attention to just these.

Let us briefly highlight some of the main known qualitative properties of global solutions to (\ref{eqCauchy}),  i.e., those that are defined for all $t>0$.
For $1<p \leq p_F=1+2/N$ there are no
global positive  solutions (see   \cite[Sec. 9.4]{Evans} and \cite[Sec. II.18]{qs}). If $1<p<p_S$, where
\[p_S=\frac{N+2}{N-2}\ \ \textrm{if}\ N\geq 3,\ \ p_S=\infty \ \textrm{if}\ N=1,2,\] stands for the critical Sobolev exponent, then as $t\to \infty$ positive global solutions converge to zero (uniformly in $\mathbb{R}^N$) at least if their initial condition  $u_0\geq 0$
is  square integrable (see \cite{Sou}). As shown in \cite[Prop. 3.3]{pg}, if $u_0\geq 0$ is bounded, convergence of global solutions to zero follows if $p\in (1,p_S)$ is such that the parabolic PDE does not possess positive, bounded solutions on $\mathbb{R}^N\times \mathbb{R}$ (see also Subsection \ref{subsubQuasi} below for this point of view). Remarkably, the latter \emph{Liouville property} was proven for $p$ in the full subcritical  range $(1,p_S)$ only very recently in \cite{quinteras2020}, where we also refer for the fascinating history of this important result. In the critical case $p=p_S$, under certain technical assumptions, it was shown recently in \cite{MM} that radial global solutions converge in the energy space (see (\ref{eqEnergia22}) below) either to zero or  to a so called 'tower of bubbles' with the adjacent bubbles having opposite sign (see also \cite{delCIME} for the construction of such solutions).
 On the other hand, very little seems to be known concerning the asymptotic behaviour of solutions in the case of  supercritical exponent $p>p_S$ (especially when dealing with nonradial solutions).
 To illustrate a major difficulty that arises in this regime, let us mention that by Pohozaev's identity the (nontrivial) radial steady state solutions, which exist only if $p\geq p_S$, have finite energy
 \begin{equation}\label{eqEnergia22}
   E(u)=\int_{\mathbb{R}^N}^{}\left\{\frac{|\nabla u|^2}{2}-\frac{|u|^{p+1}}{p+1} \right\}dx
 \end{equation}
  only if
 $p=p_S$.
As is well known, when $E$ is finite along a  solution of (\ref{eqCauchy}) it serves as a Lyapunov functional. The invariance principle \cite{H} then implies that if such solutions are global and bounded,
they are \emph{quasiconvergent} in the sense that their $\omega$-limit set consists of steady states.

\subsubsection{The radial steady states for $p\geq p_c$}\label{subsubRadialpc}Our motivation for the current work comes from the paper \cite{py}, where the authors consider the case where the exponent $p$ is strictly larger than the Joseph-Lundgren exponent
\begin{equation}\label{eqJL}
  p_c=\left\{\begin{array}{ll}
               \frac{(N-2)^2-4N+8\sqrt{N-1}}{(N-2)(N-10)}=1+\frac{4}{N-4-2\sqrt{N-1}}
&\textrm{if}\ N> 10, \\
               \infty &\textrm{if}\ N\leq 10,
             \end{array}
   \right.
\end{equation}
and the initial condition $u_0$ is restricted to be nonnegative.

In order to describe their results, and for future purposes, let us briefly recall mainly from \cite{gui,wang} some basic facts for   the   steady state problem:
\begin{equation}\label{eqSS}
  \Delta u+u^p=0,\ u>0, \ x\in \mathbb{R}^N,\ \textrm{with}  \ p\geq p_c.
\end{equation}
There is a one-parameter family $\{\varphi_\alpha,\ \alpha>0\}$ of radially symmetric solutions to (\ref{eqSS}), given by
\begin{equation}\label{eqRescal}
\varphi_\alpha(x)=\alpha \Phi\left(\alpha^{(p-1)/2}|x| \right),
\end{equation}
where $\Phi=\Phi(r)$, $r=|x|$, is the (unique) radial steady state with $\Phi(0)=1$; it is decreasing in $r>0$ and satisfies $\Phi(r)\to 0$ as $r\to \infty$.
Moreover, the following monotonicity property holds:
\begin{equation}\label{eqMonot}
  \frac{\partial}{\partial\alpha}\varphi_\alpha(r)>0,\ r\geq0,\ \alpha>0,
\end{equation}
 (see also  \cite[Lem. 2.3]{fila}).
In fact, as $\alpha\to \infty$, the solution $\varphi_\alpha$ converges  to the singular steady state
\begin{equation}\label{eqSingular}
  \varphi_\infty(x)=L|x|^{-2/(p-1)}\ \ \textrm{with}\ \ L=\left(\frac{2}{p-1}\left(N-2-\frac{2}{p-1}\right) \right)^{1/(p-1)}.
\end{equation}
We also note that   for each $\alpha>0$ one has
\begin{equation}\label{eqAs}
  \varphi_\alpha(x)=\left\{\begin{array}{ll}
                             \frac{L}{|x|^m}+a\frac{1}{|x|^{m+\lambda_1}}+o\left(\frac{1}{|x|^{m+\lambda_1}}\right) & \textrm{if}\ p>p_c, \\
                               &   \\
                             \frac{L}{|x|^m}+b\frac{ \ln |x|}{|x|^{m+\lambda_1}}+O\left(\frac{1}{|x|^{m+\lambda_1 }}\right) & \textrm{if}\ p=p_c,
                           \end{array}
\  \ \ \textrm{as}\ |x|\to \infty,
\right.\end{equation}
where \begin{equation}\label{eqmm}
        m=\frac{2}{p-1},
      \end{equation} $\lambda_1>0$ is as in (\ref{eqlambda1}) below, and $a,b<0$ are constants that depend on $\alpha>0$ (see also  \cite[Prop. 2.3]{gnw22}, and \cite[Prop. 2.1]{Seki} for the explicit dependence of $a,b$ on $\alpha$).
We note that the kernel of the linearization of (\ref{eqSS}) on the singular solution $\varphi_\infty$, when restricted to the radial class, is spanned by
\begin{equation}\label{eqKernelSing}
  \begin{array}{ll}
     \frac{1}{r^{m+\lambda_1}}\ \textrm{and}\ \frac{1}{r^{m+\lambda_2}}\ \textrm{for\ a}\ \lambda_2>\lambda_1 & \textrm{if}\ p>p_c; \\
       &   \\
     \frac{1}{r^{m+\lambda_1}}\ \textrm{and}\ \frac{\ln r}{r^{m+\lambda_1}} & \textrm{if}\ p=p_c.
   \end{array}
\end{equation}
Actually, $\lambda_2$ coincides with $\lambda_1$ if $p=p_c$.
For future reference, let us note that
\begin{equation}\label{eqLambda12}
\lambda_1>2,\end{equation}
see Remark \ref{remLambdas} below.

The monotonicity property (\ref{eqMonot}) implies easily by the maximum principle that the radial (regular) steady states are (locally) stable equilibria for (\ref{eqCauchy}) under compactly supported perturbations (with respect to the uniform norm). The stability of these equilibria under general perturbations has been established with respect to suitable weighted uniform norms in \cite{gui,gnw22,wang}. On the other hand, a corresponding stability property with respect to the usual $L^2$ norm has been obtained in \cite{PYA} for initial conditions that lie below the radial singular steady state in absolute value. We note in passing that analogous positive radially symmetric steady states
also exist if $p\in (p_S,p_c)$, however they are unstable 'in any reasonable sense' (see the aforementioned references for more details). \subsubsection{On the radial symmetry of positive steady states for $p>1$}\label{subsubRadialp1} For future purposes, let us mention that solutions of (\ref{eqSS}) with $p>p_S$ are radially symmetric about some point if they satisfy
\begin{equation}\label{eqZOU}
  |x|^mu(x)\to L\ \ \textrm{as}\ \ |x|\to \infty
\end{equation}
 for  $p\geq \frac{N}{N-4}$ ($N\geq 5$) and $p\in \left(p_S,\frac{N+1}{N-3}\right)$, see \cite{GUO} and \cite{ZOU} respectively. For $p\in \left[\frac{N+1}{N-3},\frac{N}{N-4} \right)$ ($N\geq 4$), it was shown in the former reference that   they are radially symmetric about some point
if and only if they satisfy (\ref{eqZOU}) and
\begin{equation}\label{eqZOU22}
|x|^{\frac{N}{2}-m-1}\left(|x|^mu(x)-L \right)=0\ \ \textrm{as}\ |x|\to \infty.
\end{equation}

 We emphasize that the positivity of $u$ played a crucial role in the proofs of the aforementioned
results since they relied on the method of \emph{moving planes} (see for instance the review \cite{pacella}), after establishing an accurate asymptotic expansion for $u$ at infinity (see also \cite{kelei}). For completeness, let us point out that (\ref{eqSS}) is fully understood if $1<p\leq p_S$ (again by the method of moving planes, see \cite{ChenLi}). More specifically,  there are no solutions if $1<p<p_S$, while if $p=p_S$ all solutions are radially symmetric about some point and can be given explicitly (the so called 'bubbles').

\subsubsection{Quasiconvergence and the Liouville property}\label{subsubQuasi}
We can now present the results of \cite{py} which motivated the current work.

\begin{thm}\label{thm1}\cite{py}
Assume that $p>p_c$ and let $u_0\in C(\mathbb{R}^N)$ satisfy
\begin{equation}\label{eqOrder}
  \varphi_\alpha(x)\leq u_0(x)\leq \varphi_\beta(x),\ \ x\in\mathbb{R}^N,
\end{equation}
for some $0<\alpha<\beta\leq \infty$. Then
\[
\omega(u_0)\subseteq\{\varphi_\gamma\ :\ \alpha\leq \gamma \leq \beta \},
\]
where $\omega(u_0)$ stands for the $\omega$-limit set of the solution $u(\cdot,t,u_0)$ of (\ref{eqCauchy}):
\[
\omega(u_0)=\left\{\phi\ : \ u(\cdot, t_n,u_0)\to \phi \ \textrm{for some sequence}\ t_n\to \infty \right\},
\]
with the convergence taking place in the supremum norm.
\end{thm}

We insist that the above result is nontrivial since the usual Lyapunov functional could be infinite along $u$ (due to its slow algebraic spatial decay).
Actually, the main effort in the aforementioned reference was placed in establishing the following \emph{Liouville type property}, which implies directly the above \emph{quasiconvergence}
result (see also the discussion that follows).
\begin{thm}\cite{py}\label{thm2}
If $p>p_c$, $0<\alpha<\beta\leq \infty$, and $u$ is a (classical) solution of
\begin{equation}\label{eqEq}
  u_t=\Delta u+|u|^{p-1}u
\end{equation}
on $\mathbb{R}^N\times (-\infty,0]$,
satisfying
\begin{equation}\label{eqSandwich9}
  \varphi_\alpha\leq u(\cdot,t)\leq \varphi_\beta\ \ \textrm{for all}\ t<0,
\end{equation}
then $u\equiv \varphi_\gamma$ for some $\gamma>0$.
\end{thm}
As was mentioned in \cite{py}, the solution of the Cauchy problem (\ref{eqCauchy}) with initial condition as in (\ref{eqOrder}) exists globally and satisfies (\ref{eqSandwich9}) for all $t>0$. Therefore, one may assume that the conditions in Theorem \ref{thm2} hold for  $t\in \mathbb{R}$; such solutions are frequently called \emph{eternal}. In fact, by the invariance principle \cite{H}, the Cauchy problem for (\ref{eqEq}) with initial condition in $\omega(u_0)$, where $u_0$ is as in (\ref{eqOrder}), admits an eternal solution that satisfies (\ref{eqSandwich9}) for all $t\in \mathbb{R}$. So, Theorem \ref{thm1} follows directly from Theorem \ref{thm2}.

To    prove the latter, the authors of \cite{py}  first reformulated the problem in self-similar variables as in \cite{GK} (the so called Type I scaling). In these variables, each $\varphi_\alpha$ with $0<\alpha<\infty$ corresponds to a monotone heteroclinic connection between $\varphi_\infty$ and zero. The two sided bound (\ref{eqSandwich9}) with $\beta<\infty$ implies that $u$ in these variables is squeezed between two (time) translations of the aforementioned heteroclinic, and thus is a heteroclinic connection between $\varphi_\infty$ and zero itself. Theorem \ref{thm2} with $\beta<\infty$ then boils down to showing that such a heteroclinic connection is unique up to translations. To this end, they analyzed carefully the projections of the solution on the stable, unstable and center eigenspaces of the linearized operator around the singular steady state $\varphi_\infty$ as the (rescaled) time approaches $-\infty$. Since they focus only on this limit, they can allow the case $\beta=\infty$ in the assumptions. We point out that, even though the space domain is unbounded, the aforementioned linear operator is self-adjoint with compact resolvent (in the natural weighted $L^2$ space) thanks to the assumption  $p>p_c$ (cf. \cite{HV}). In passing, let us note that in the critical case $p=p_c$ the aforementioned property was  proven only recently in \cite[Prop. 2.2]{Seki}.   For the local nonlinear analysis near $\varphi_\infty$ they used  as a guideline the invariant manifold theory from infinite-dimensional dynamical systems in the spirit of  \cite{fk,mz} (see also the survey \cite{fm}).

 It is worth mentioning that they  expressed their belief that the conclusion of Theorem \ref{thm1} (with the obvious modifications) is likely to be valid
for all solutions that are bounded in absolute value by $\varphi_\infty$ and  also for
$p=p_c$.

We refer to the recent papers \cite{polacikQui} and \cite{quinteras2020} for some important new developments on Liouville type
theorems for (\ref{eqEq})  and an overview. We highlight that by combining the invariant manifold ideas and the intersection number argument (zero-number principle) for one-dimensional parabolic equations, it was shown in the former reference that
positive, radial and bounded eternal solutions of (\ref{eqEq}) are steady states at least if $p>p_L$ (where $p_L>p_c$ stands for Lepin's critical exponent).

Remarkably, it was shown in \cite{filaYana} that (\ref{eqEq}) admits a positive solution of homoclinic type that converges to zero as $t\to \pm \infty$ and is radial in space (decaying to zero), provided that $p\in (p_S,p_L)$. We point out that for each sufficiently large negative time this solution has precisely two (positive) radial intersections with the singular steady state $\varphi_\infty$
which are eventually lost.
On the other hand,        it was also proven in the same reference that the only solution of (\ref{eqEq}) with $p\in [p_S,p_c)$ in $\mathbb{R}^N\times (-\infty,0)$ that satisfies $|u|\leq \varphi_\infty$ is the trivial one (see also \cite{sourggrok} for a simple proof and  extensions of this result).

\subsection{Our results}\label{subRes} The above discussion motivated our main result, which is the following.
\begin{thm}\label{thmMine}
Suppose that $p\geq p_c$ and that $u$ is a (classical) solution of (\ref{eqEq}) in $\mathbb{R}^N\times \mathbb{R}$
satisfying
\begin{equation}\label{eqSandwich}
   |u(\cdot,t)|\leq \varphi_\beta\ \ \textrm{for all}\ t\in \mathbb{R}
\end{equation}
for some
           $\beta\in  (0,\infty] $,
                                                 and
\begin{equation}\label{eqlower}
  u(x,t)=
                  \frac{L}{|x|^m}+o\left(\frac{1}{|x|^{m+\lambda_1-2}}\right)  \end{equation}
or
\begin{equation}\label{eqlower222}
  u(x,t)=
                  \frac{L}{|x|^m}+o\left(\frac{\ln|x|}{|x|^{m+\lambda_1-2}}\right) \ \ \textrm{if}\ p=p_c\ \textrm{and}\ \beta<\infty,
\end{equation}
uniformly in $t\in \mathbb{R}$,  as $|x|\to \infty$  (where $m,\lambda_1>0$ are as in (\ref{eqKernelSing}), recall also (\ref{eqLambda12})).
Then, the solution $u$ does not depend on time.
\end{thm}

We refer to Remark \ref{remGibbons} for a corresponding rigidity result for the (ill posed) hamiltonian counterpart to (\ref{eqEq}) that can be proven along the same lines.

Unexpectedly,  in the critical case $p=p_c$, where the proof of \cite{py} does not apply, we can do even better if $\beta<\infty$ (compare (\ref{eqlower}) to (\ref{eqlower222})).

Let us briefly relate our main result to Theorem \ref{thm2}.    If $u$ satisfies the assumption (\ref{eqSandwich9}) of the latter, then it clearly satisfies our assumption (\ref{eqlower}) (recall (\ref{eqAs}) and the comment immediately following Theorem \ref{thm2}).
In particular,   our result yields the following noteworthy corollary.

\begin{cor}\label{cor00}
 The validity of Theorem \ref{thm2} holds for the critical exponent case $p=p_c$ as well.
 \end{cor}

  Let us note in this regard that, once the solution of Theorem \ref{thm2} is shown to be independent of time, its radial symmetry follows from
the elliptic results that were mentioned in relation to (\ref{eqZOU}) and (\ref{eqZOU22}); observe also that $u$ must satisfy (\ref{eqZOU22}) because of the fact that $\varphi_\alpha$ and $\varphi_\beta$ do so (by the necessity of that condition). We point out that these radial symmetry results were
not needed with the approach of \cite{py}.

Let us mention an important consequence of our main result.
In light of the above discussion, our result implies the following extension of Theorem \ref{thm1} which confirms  part of the expressed belief in \cite{py} that was mentioned in the previous subsection. \begin{cor}\label{corGidas}
                                                                                                              The conclusion of Theorem \ref{thm1} holds even if $p=p_c$.\end{cor}

%\begin{rem}\label{remPacella}
%There is a well developed machinery, based on the method of moving planes, for establishing the radial symmetry of positive solutions
%of the steady state problem (see \cite{pacella}  and the references therein). However, it is not clear to us how to adapt this approach to give
%a direct proof of the radial symmetry
%of solutions of (\ref{eqSS}) that satisfy (\ref{eqSandwich9}).
%\end{rem}

The radial symmetry properties of    solutions to (\ref{eqSS})  is actually a hot topic   of research these days, see for example \cite[Sec. 11]{weiChan}.
In this regard, we can establish the following result. This theorem gains in interest if we realize that we do not assume that the solution is positive (keep in mind  the comment after (\ref{eqZOU22})).
\begin{thm}\label{thmMine2}
Suppose that $u\in C^2(\mathbb{R}^N)$ satisfies \[\Delta u+|u|^{p-1}u=0, \ x\in \mathbb{R}^N, \ \textrm{with}\ p\geq p_c,\] the bound (\ref{eqSandwich})  with $\beta=\infty$, and the asymptotic behaviour
(\ref{eqAs}) for some $a\leq 0$, $b<0$. Then, $u$ is radially symmetric. Thus, thanks to (\ref{eqRescal}), $u\equiv \varphi_\gamma$ for some $\gamma\in (0,\infty)$.
\end{thm}

\subsection{Method of proof}\label{subsecMethod}
Our method of proof of Theorem \ref{thmMine} is based on the maximum principle and a sweeping type argument,
exploiting the presence of
the simply ordered curve of equilibria $\varphi_\alpha$. This may come as a surprise at first, as this approach was ruled out in \cite{py} (see the discussion in their introduction in relation to the theory of monotone dynamical systems).
However, we stress that we apply this argument to the linearized equation for $u_t$, instead of applying it in (\ref{eqEq}). The main observation is that $u_t$ goes to zero faster, as $|x|\to \infty$, than the \emph{positive} element of the kernel of the
linearization of the steady state problem that is generated by differentiation with respect to  $\alpha$; if '$\alpha=\infty$', keeping in mind (\ref{eqKernelSing}), the corresponding element in the kernel is plainly $r^{-m-\lambda_1}$ (even if $p=p_c$, since the other element in that case is sign changing). We emphasize that this property allows us to 'sweep' with this element which, due to (\ref{eqSandwich}), is a \emph{positive supersolution} to the linearized equation of (\ref{eqEq}) on $u$.
The unboundedness issue in the $t$ direction can be tackled by exploiting that the equation (\ref{eqEq}) and the bound (\ref{eqSandwich}) are invariant under
time translations.

A related idea  can be found in our paper \cite{ss} on an elliptic system arising in the study of phase separation in Bose-Einstein condensates, where a sweeping argument was applied to the tangential  derivatives to the (blown-up) interface of the solutions.

We prove Theorem \ref{thmMine2} analogously, by applying a sweeping argument to show that  the angular derivatives
of $u$ are identically equal to zero.

In contrast, we refer to \cite{BH} where a sweeping argument (of sliding type) was applied directly to a class of parabolic
Allen-Cahn type  equations to prove the Liouville property for eternal solutions that stay between two ordered standing wave solutions
(compare with (\ref{eqSandwich9}) herein). Actually, a sweeping argument in the linearized equation may also be applied successfully in that setting as well (after the solution has been shown to be monotone in one spatial direction by, say, the method of moving planes). In relation to Theorem  \ref{thmMine2},  it is worth mentioning that a sweeping argument of sliding type has been applied in \cite[Lem. 3.2]{busca} directly to the nonlinear equation $\Delta u+f(u)=0$ in $\mathbb{R}^N$ to show that it cannot have ordered positive solutions such that $u\to 0$ as $|x|\to \infty$, provided that $f\in C^1$ satisfies $f(0)=0$ and $f'(0)<0$ (we also refer to \cite{sourdisCRM} for an extension of this result to sign-changing solutions using the corresponding parabolic flow).

Lastly, in order to highlight the flexibility of our approach, we refer to
\cite{sourdisProceedings}, where a one-dimensional symmetry result for eternal solutions to the Fisher-KPP equation
was  proven using the proof of Theorem \ref{thmMine} as a guideline.

\subsection{Open problems and a connection to ancient solutions of the mean curvature flow}\label{subsecOpen}Our Theorem \ref{thmMine} confirms to a large extend the expectation, first expressed in \cite{py},
that the Liouville property holds for any eternal solution of (\ref{eqEq}) with $p\in [p_c,\infty)$ that is bounded in absolute value
by $\varphi_\infty$. The only thing that remains is to show that such solutions have to satisfy the  asymptotic behaviour (\ref{eqlower}) (up to a sign change, if they are nontrivial). For this purpose,
we believe that the methods that have been developed for the study of the asymptotic behaviour as $|x|\to \infty$ of  the steady states  could be useful (see \cite{verong} and also Remark \ref{remWang} herein).

It is known that Gelfand's equation
\[
u_t=\Delta u+e^u,\ \ x\in \mathbb{R}^N,\ t\in \mathbb{R},
\]
admits a one-parameter family of  ordered, radial steady states when $N\geq 10$ (see \cite{expog}).
We expect that analogous Liouville type theorems are valid for this famous equation as well.

The authors of \cite{cabreLNM} and \cite{velaz} have made a remarkable analogy between the radial solutions of (\ref{eqSS}) and the minimal surfaces that foliate
the Simons cone in space dimension larger or equal than 8. The analog of the Simons cone is the singular solution $\varphi_\infty$. The same heuristic analogy can also be made with Lawson's cones in \cite{Davini} and the references therein.
On the other hand, the aforementioned invariant manifold ideas for obtaining  Liouville type theorems for (\ref{eqEq}) also play an important role in the study
of ancient solutions to the mean curvature flow (see  \cite{ads} and the references therein). In light of the above, we believe that it would be of interest to explore whether analogous Liouville type results
to those in Theorems \ref{thm2}  and \ref{thmMine} hold for ancient solutions to the mean curvature flow.

This intuitive connection to mean curvature flow can be further supported    by comparing
the matched asymptotic analysis in \cite{Amat}, concerning symmetric self-shrinking solutions for the mean curvature flow in low dimensions, to
that in \cite{Cmat} concerning radial backwards self-similar solutions of  (\ref{eqEq}) with $p\in (p_S,p_c)$.

\subsection{Outline of the paper}\label{subsecOutline}The rest of the paper is   primarily devoted to the proofs of our main results in the next section.
More specifically, in Subsection \ref{subsec1} we will derive bounds for the asymptotic behaviour as $|x|\to \infty$ of  some partial derivatives of the solution in Theorem \ref{thmMine}. Actually, the lemma that we will prove to accomplish this task will be frequently  used for related purposes in the paper.
In Subsection \ref{subsec2}, relying on this information, we will prove Theorem \ref{thmMine}. As we have already explained,  the validity of Corollaries \ref{cor00} and \ref{corGidas}  follows at once from Theorem \ref{thmMine}. Therefore, their proofs will be omitted. In the same subsection we will also remark on the  elliptic rigidity  result that we hinted at following the statement of Theorem \ref{thmMine}.  Lastly, in Subsection \ref{subsec3} we will prove Theorem \ref{thmMine2} and provide some related remarks. In particular, a further heuristic connection to minimal surface theory will come up in Remark \ref{remWang}.

\section{Proofs of the main results}
In this section we will prove our results and provide some related remarks.
\subsection{Some preliminary estimates}\label{subsec1}Our objective in this subsection is to estimate $u_t$ for $|x|\gg 1$ uniformly in $t\in \mathbb{R}$.

In the sequel we will make repeated use of the  following interpolation lemma, which represents a direct extension to the parabolic setting of Lemma A.1 in \cite{bbh} concerning elliptic equations (see also (4.45) in \cite{GT}).
It is of independent interest and may be useful in other problems.
\begin{lem}\label{lemBBH}
Assume that $\psi$ satisfies (classically)
\[
\psi_t-\Delta \psi=f\ \textrm{in}\ \mathcal{C}_R=\left\{|x|<R,\ |t|<R^2\right\}\subset \mathbb{R}^N\times \mathbb{R}.
\]
Then
\[
|\nabla_x \psi(0,0)|^2\leq C\|f\|_{L^\infty(\mathcal{C}_R)}\|\psi\|_{L^\infty(\mathcal{C}_R)}+\frac{C}{R^2}\|\psi\|_{L^\infty(\mathcal{C}_R)}^2
\]
for some constant $C>0$ that depends only on the dimension $N$.
\end{lem}
\begin{proof}
  The proof follows closely that of the aforementioned reference \cite{bbh}. Essentially the only difference is that one has to employ the corresponding interior parabolic regularity estimates
  instead of the elliptic ones.

   Let us present some details for the reader's convenience.
Let $\lambda\leq R$ be a positive constant to be determined. The rescaled
function
\[w(y,\tau) = \psi(\lambda y,\lambda^2\tau)\]
is well defined in the space-time cylinder $\mathcal{C}_1=\{|y|<1,\ |\tau|<1\}$ (since $\lambda \leq R$) and  satisfies
\[w_\tau-\Delta_y w=\lambda^2f(\lambda y,\lambda^2 \tau)\ \textrm{in}\ \mathcal{C}_1.\]
From standard interior parabolic estimates  there exists a constant $C>0$, depending only on $N$ and $1<p<\infty$, such that
\[
\|w\|_{W^{2,1}_p(\mathcal{C}_{1/2})}\leq C\left( \lambda^2\|f(\lambda y,\lambda^2 \tau)\|_{L^p(\mathcal{C}_1)}
+\|w\|_{L^p(\mathcal{C}_1)} \right)
\]
 (see for instance \cite[Thm. 7.22]{L} or \cite[Thm. 48.1]{qs}). On the other hand, by the parabolic Sobolev embedding, we have\[ \|w\|_{C^{1+\theta,\frac{1+\theta}{2}}(\mathcal{C}_{1/2})}\leq C \|w\|_{W^{2,1}_p(\mathcal{C}_{1/2})} \ \ \textrm{with}\ \theta=1-\frac{N+2}{p}\ \textrm{if}\ N+2<p<\infty, \] for a possibly larger constant $C>0$ still depending only on $N$ and $p$ (see \cite[pgs. 80, 342]{lad} and also \cite[App. E]{flemg}). By the above two relations, we deduce that
\[|\nabla_y w(0,0)| \leq C\left(\lambda^2\|f(\lambda y ,\lambda^2 \tau)\|_{L^\infty(\mathcal{C}_1)} + \|w\|_{L^\infty(\mathcal{C}_1)}\right),\]
where the above constant  $C>0$ depends only on the space dimension $N$. In  the original variables, the above estimate takes the form
\[
\lambda |\nabla_x  \psi(0,0)| \leq C\left(\lambda^2\|f\|_{L^\infty(\mathcal{C}_R)} + \|\psi\|_{L^\infty(\mathcal{C}_R)}\right).
\]
The rest of the proof runs as in Lemma A.1 in \cite{bbh}, where one divides both sides of the above inequality by $\lambda$ and then minimizes the resulting righthand side  over  $\lambda \in (0,R]$.
\end{proof}

The above lemma will help us to make rigorous the following natural expectation (keep in mind (\ref{eqSingular})).
\begin{lem}\label{lemExp}
If $u$ is as in Theorem \ref{thmMine}, then
\[
\Delta u(x,t)=
                  -\frac{L^p}{|x|^{pm}}+o\left(\frac{1}{|x|^{m+\lambda_1}}\right)\   \   \textrm{in the case of}\ (\ref{eqlower}),
\]
or\[
\Delta u(x,t)=
                  -\frac{L^p}{|x|^{pm}}+o\left(\frac{\ln|x|}{|x|^{m+\lambda_1}}\right) \ \ \textrm{if}\ p=p_c\ \textrm{and}\ \beta<\infty,\   \   \textrm{in the case of}\ (\ref{eqlower222}),
\]
uniformly in $t\in \mathbb{R}$,  as $|x|\to \infty$.
\end{lem}
\begin{proof}
We will only consider the case where $u$ satisfies (\ref{eqlower}), as the other case   can be handled in a similar fashion.

  Let $\epsilon>0$ be arbitrary.
According to (\ref{eqlower}), recalling the definition of the singular solution $\varphi_\infty$ from (\ref{eqSingular}), we can write
\begin{equation}\label{eqPsi1} u(x,t)=                  \varphi_\infty(x)+\psi(x,t),\end{equation}
where $\psi$ satisfies
\begin{equation}\label{eqPsi2}
   |\psi(x,t)|\leq \frac{\epsilon}{|x|^{m+\lambda_1-2}} \ \ \textrm{if}\ \ |x|\geq K,\ t\in \mathbb{R},
\end{equation}for some $K=K(\epsilon)>0$.
We then have
\begin{equation}\label{eq11}
  \psi_t-\Delta \psi=|u|^{p-1}u-\varphi_\infty^p\stackrel{(\ref{eqSandwich})}{=}O(1)p\varphi_\infty^{p-1}\psi=O(|x|^{-2})\frac{\epsilon}{|x|^{m+\lambda_1-2}}=\frac{O(1)\epsilon}{|x|^{m+\lambda_1}},
\end{equation}
for $|x|\geq K$, $t\in \mathbb{R}$. We point out that throughout the proof Landau's symbol $O(1)$ will be independent of small $\epsilon$.
 We can now apply Lemma \ref{lemBBH} with center $(x,t)$ (instead of $(0,0)$) and $R=|x|/2$ to obtain
\[
|\nabla_x\psi(x,t)|\leq \frac{C\epsilon}{|x|^{m+\lambda_1-1}},\ \ |x|\geq 2K,\ t\in \mathbb{R},
\]
for some constant $C>0$ that is independent of small $\epsilon>0$.
Actually, abusing notation slightly, in the sequel  we will denote by $C$ a generic positive constant whose value may increase as the proof progresses.

Let \[z_i=\partial_{x_i}\psi \ \ \textrm{for} \ \ i=1,\cdots,N.\]
Dropping the subscript $i$ for notational simplicity,  it follows from the first equality in (\ref{eq11}) that
\[\begin{array}{rcl}
    z_t-\Delta z & = & p|u|^{p-1}(\partial_{x_i}\varphi_\infty+z)-p\varphi_\infty^{p-1}\partial_{x_i}\varphi_\infty \\
      &   &   \\
      & = & p\left(|u|^{p-1}-\varphi_\infty^{p-1}\right)\partial_{x_i}\varphi_\infty+p|u|^{p-1}z\\
     &  & \\
      & =  &\frac{(p-1)O(1)}{|x|^{m(p-2)}}\frac{\epsilon}{|x|^{m+\lambda_1-2}}\frac{1}{|x|^{m+1}}  +\frac{O(1)}{|x|^{m(p-1)}}\frac{\epsilon}{|x|^{m+\lambda_1-1}}\\
      &&\\
      &\stackrel{(\ref{eqmm})}{=}&  \frac{O(1)\epsilon}{|x|^{m +\lambda_1+ 1}},
  \end{array}
\]
for $|x|\geq 2K$, $t\in \mathbb{R}$.
Arguing as before, we find
\[
    |\nabla_x z(x,t)|   \leq   \frac{C\epsilon}{|x|^{ m+ \lambda_1}}
 \ \ \textrm{for}\  |x|\geq 4K,\  t\in \mathbb{R},
  \]
which implies the assertion of the lemma.
\end{proof}

The main information that we will need from this subsection for the proof of Theorem \ref{thmMine} is encapsulated in the following corollary.
\begin{cor}\label{cor1}
  If $u$ is as in Theorem \ref{thmMine}, then
  \[
 u_t(x,t)=
                  o\left(\frac{1}{|x|^{m+\lambda_1}}\right)\   \   \textrm{in the case of}\  (\ref{eqlower}), \]
or
\[
 u_t(x,t)=
                  o\left(\frac{\ln|x|}{|x|^{m+\lambda_1}}\right)  \ \ \textrm{if}\ p=p_c\ \textrm{and}\ \beta<\infty,\   \   \textrm{in the case of}\  (\ref{eqlower222}),
\]
as $|x|\to \infty$, uniformly in $t\in \mathbb{R}$.
\end{cor}
\begin{proof}
We will only deal with the case where (\ref{eqlower}) holds, as the case of (\ref{eqlower222}) can be handled  analogously.

Recalling (\ref{eqPsi1}) and (\ref{eqPsi2}), we obtain from Lemma \ref{lemExp} that
\[\begin{array}{rcl}
    u_t & = & \Delta u+|u|^{p-1}u \\
      &   &   \\
      & = & -\varphi_\infty^p+o\left(\frac{1}{|x|^{m+\lambda_1}}\right)+ |\varphi_\infty+\psi|^{p-1}(\varphi_\infty+\psi) \\
      &   &   \\
      & = & O(1)p\varphi_\infty^{p-1}\psi+o\left(\frac{1}{|x|^{m+\lambda_1}}\right) \\
      &   &   \\
     & = &O\left(\frac{1}{|x|^2} \right) \frac{o(1)}{|x|^{m+\lambda_1-2}}+o\left(\frac{1}{|x|^{m+\lambda_1}}\right)\\
      &&\\
      &=&o\left(\frac{1}{|x|^{m+\lambda_1}}\right),
  \end{array}
\]
as $|x|\to \infty$, uniformly in $t\in \mathbb{R}$, as desired.
\end{proof}
\subsection{Proof of Theorem \ref{thmMine}}\label{subsec2}
\begin{proof}We will only deal with the case $p>p_c$, as the case of equality can be handled  analogously.

For the reader's convenience, we will first treat the case $\beta<\infty$. In fact,
without loss of generality, we will assume that $\beta=1$.

We begin with a preliminary observation that will be useful in the sequel. Since $u$ is a bounded solution of (\ref{eqEq}) in $\mathbb{R}^N\times \mathbb{R}$,  standard
interior estimates for linear parabolic equations \cite{lad,L,WangM} and Sobolev embeddings imply that \begin{equation}\label{equB}
                                            u \ \textrm{is bounded in}\ C^{2+\theta,1+\theta/2}(\mathbb{R}^N\times \mathbb{R})\ \textrm{for any}\ \theta\in (0,1).
                                          \end{equation}
In this regard, we point out that  $|u|^{p-1}u$ plays the role of the inhomogeneous term and the aforementioned estimates are applied in space-time cylinders of the form $\mathcal{C}_1({x,t})=\left\{ (y,\tau)\ :\ |y-x|<1,\ |\tau-t|<1\right\}$; one first applies the $W_p^{2,1}$ estimates, then the (elliptic) Sobolev Morrey embedding, and finally the $C^{2+\theta,1+\theta/2}$ estimates.

  The assertion of the theorem is  equivalent to
  \begin{equation}\label{eqvdef}
    v:=u_t
  \end{equation}being identically equal to zero. The rest of the proof is devoted to showing the latter property.

  Clearly $v$ satisfies the linearized equation
  \begin{equation}\label{eqv}
    v_t=\Delta v+p|u|^{p-1}v,\ \ x\in \mathbb{R}^N,\ t\in \mathbb{R}.
  \end{equation}
By virtue of (\ref{equB}),  $v$ is bounded in $\mathbb{R}^N\times \mathbb{R}$. In turn, arguing as before, standard parabolic estimates and Sobolev embeddings
yield \begin{equation}\label{eqvB}
                                            v \ \textrm{is bounded in}\ C^{2+\theta_0,1+\theta_0/2}(\mathbb{R}^N\times \mathbb{R})\ \textrm{for some}\ \theta_0\in (0,1),
                                          \end{equation} (we note that the above relation holds for any $\theta_0\in (0,1)$ if $p\geq 2$ or $u$ does not change sign).
  Moreover, on account of Corollary \ref{cor1}, we have
\begin{equation}\label{eqUpper}
v(x,t)=
                  o\left(\frac{1}{|x|^{m+\lambda_1}}\right)  \ \ \textrm{as}\  |x|\to\infty,\  \textrm{uniformly in}\ t\in \mathbb{R}.
\end{equation}

A crucial role will be played by the linearization of (\ref{eqSS}) on the radial steady state $\varphi_1=\Phi$ which bounds $|u|$. From (\ref{eqRescal}) and (\ref{eqMonot}), it follows that the radially symmetric function
\[
Z:=\Phi+\frac{1}{m}x\cdot \nabla \Phi
\]
is   \emph{positive} and belongs to the kernel of the aforementioned linearized operator, i.e., it satisfies
\begin{equation}\label{eqZ}
  \Delta Z+p\Phi^{p-1}Z=0,\ \ x\in \mathbb{R}^N.
\end{equation}
Furthermore, according to (\ref{eqAs}) (which can be differentiated in the obvious way), we have
\begin{equation}\label{eqAsZ}
  Z(x)=\left\{\begin{array}{ll}
                \frac{\lambda_1}{m}\frac{|a|}{|x|^{m+\lambda_1}}+o\left(\frac{1}{|x|^{m+\lambda_1}}\right)  & \textrm{if}\ \ p>p_c, \\
                  &   \\
               \frac{\lambda_1}{m}\frac{|b|\ln |x|}{|x|^{m+\lambda_1}}+O\left(\frac{1}{|x|^{m+\lambda_1}}\right)  & \textrm{if}\ \ p=p_c,
              \end{array}
  \right.\ \ \textrm{as}\ |x|\to \infty,\end{equation}
(see also \cite[Lem. 2.3]{fila}).

The importance of $Z$ for our purposes is that it serves as a positive
 \emph{super-solution} of (\ref{eqv}). Indeed, using (\ref{eqSandwich}) (with $\beta=1$),
we get
\begin{equation}\label{eqZeq}
Z_t-\Delta Z-p|u|^{p-1}Z\stackrel{(p>1)}{\geq}-\Delta Z-p\Phi^{p-1}Z\stackrel{(\ref{eqZ})}{=}0.
\end{equation}
The key observation behind our proof is that (\ref{eqUpper}) and (\ref{eqAsZ}) give
\begin{equation}\label{eqQuot}
  \lim_{|x|\to \infty}\frac{v(x,t)}{Z(x)}=0\  \textrm{ uniformly in}\ t\in \mathbb{R}.
\end{equation}
%Indeed, thanks to (\ref{eqUpper}) and (\ref{eqAsZ}),    the above relation clearly holds.

Armed with the above information, we will  reach our goal, i.e. show that $v\equiv 0$, by adapting Serrin's sweeping principle \cite{Se} from elliptic PDEs (see also \cite{sourdis} and the references therein).
For this purpose, let us consider the set
\[
  \Lambda=\left\{\lambda\geq 0\ :\ \mu Z\geq v\ \textrm{in}\ \mathbb{R}^N\times \mathbb{R} \ \textrm{for every}\ \mu \geq \lambda\right\}.
\]
Our  objective is to show that $\Lambda = [0, \infty)$, which will yield $v\leq 0$. We can also apply this
 argument, with $v$ replaced by $-v$, to obtain  $v\geq  0$ and therefore conclude.

We first prove that $\Lambda\neq \emptyset$, and thus by continuity \begin{equation}\label{eqLambda}
                                                         \Lambda = [\tilde{\lambda}, \infty)\ \textrm{for some}\  \tilde{\lambda}\geq 0.
                                                       \end{equation}
In view of   (\ref{eqQuot}), it is evident that there exists a large $M>0$ such that
\[
 v(x,t)\leq Z(x)\ \textrm{if}\ |x|\geq M,\ t\in \mathbb{R}.
\]
Hence, recalling (\ref{eqvB}), we deduce that
\begin{equation}\label{eqLambdaBar}
v(x,t)\leq \bar{\lambda} Z(x),\ \ (x,t)\in\mathbb{R}^N\times \mathbb{R},
\
\textrm{with}
\
\bar{\lambda}=1+\frac{\sup_{\mathbb{R}^N\times \mathbb{R}}|v|}{\min_{|x|\leq M}Z}.
\end{equation}
This means   that $\bar{\lambda}\in \Lambda$, i.e. (\ref{eqLambda}) is valid for some $\tilde{\lambda}\in [0,\bar{\lambda}]$.

In order to establish that $\tilde{\lambda} = 0$, as desired, we will argue by contradiction. So let us suppose that $\tilde{\lambda} > 0$.
To show that this is absurd, by the definition of the set ${\Lambda}$ and (\ref{eqLambda})
it suffices to prove that there exists a small $\delta > 0$ such that
\begin{equation}\label{eqLambdaTilt}
  (\tilde{\lambda}-\delta)Z(x)>v(x,t),\ \ (x,t)\in\mathbb{R}^N\times \mathbb{R}.
\end{equation}

Suppose the above relation were false. Then we could find $\lambda_n <\tilde{\lambda}$ with $\lambda_n \to \tilde{\lambda}$˜, $x_n \in \mathbb{R}^N$
 and $t_n \in \mathbb{R}$
 such that
\begin{equation}\label{eqContra}
  v(x_n, t_n) \geq \lambda_nZ(x_n), \ \ n\geq 1.
\end{equation}

By virtue of (\ref{eqQuot}), and our assumption that $\tilde{\lambda}>0$, we infer that the sequence $\{x_n\}$ is bounded. Hence, passing to a subsequence if necessary, we may assume that
\begin{equation}\label{eqX}
  x_n\to x_\infty\in \mathbb{R}^N.
\end{equation}

In contrast, we claim  that the sequence $\{t_n\}$ is unbounded.
Suppose, contrary to our claim,  that $\{t_n\}$ is bounded. Then, possibly along a further subsequence, we get  $t_n\to t_\infty$ for some $t_\infty\in \mathbb{R}$. Therefore, in view of the comments before (\ref{eqContra}), and (\ref{eqX}), letting $n \to\infty$ in the former relation yields $v(x_\infty, t_\infty) \geq \tilde{\lambda}  Z(x_\infty).$ On the other hand, since  $\tilde{\lambda}\in \Lambda$ (recall (\ref{eqLambda})),
 by the definition of the latter set we assert that \begin{equation}\label{eqvleqZ}
          v(x, t) \leq \tilde{\lambda}  Z(x),\ \   (x,t)\in \mathbb{R}^N\times \mathbb{R}.
        \end{equation} Consequently, on account of (\ref{eqv}) and (\ref{eqZeq}), we deduce from the strong maximum principle for linear parabolic equations \cite{lad,L} that $v\equiv \tilde{\lambda}  Z$. However, this is not possible by (\ref{eqQuot}) and our assumption that $\tilde{\lambda}>0$. It is worth noting that another way to derive a contradiction from the last identity is to recall the definition of $v$ from (\ref{eqvdef}) and the bound (\ref{eqSandwich}) for $u$.

Let us assume that
 $t_n\to -\infty$  (up to a subsequence); the scenario $t_n\to +\infty$ can be handled in the same way ($\{t_n\}$ can actually be shown to be bounded
from above by the parabolic maximum principle similarly to \cite[Lem. 2.1 (i)]{ChenLi}). Let us now consider the time translated functions
 \[
 U_n(x,t)=u(x,t+t_n)\ \ \textrm{and}\ \ V_n(x,t)=v(x,t+t_n), \  \ n\geq 1.
 \]
 Clearly, $U_n$ continues to satisfy (\ref{eqEq}), (\ref{eqSandwich}) (with $\beta=1$), and (\ref{equB}) uniformly with respect to $n$; while $V_n$ satisfies
 \[
    V_t=\Delta V+p|U_n|^{p-1}V,\ \ x\in \mathbb{R}^N,\ t\in \mathbb{R},
  \]
  and (\ref{eqvB}) uniformly with respect to $n$.
We also note that (\ref{eqQuot}) becomes
\[
  \lim_{|x|\to \infty}\frac{V_n(x,t)}{Z(x)}=0\  \textrm{ uniformly in}\ t\in \mathbb{R}\ \textrm{and}\ n\geq 1.
\]
Moreover, from (\ref{eqContra}) and (\ref{eqvleqZ}) we obtain
\begin{equation}\label{eqRecal4}
    V_n(x_n, 0) \geq \lambda_nZ(x_n) \ \ \textrm{and}\ \
          V_n(x, t) \leq \tilde{\lambda}  Z(x),\ \   (x,t)\in \mathbb{R}^N\times \mathbb{R},
        \end{equation}respectively.

By  the aforementioned uniform H\"{o}lder estimates and a standard diagonal-compactness argument, passing to a further subsequence if necessary, we may assume that
\begin{equation}\label{eqMasviU}U_n \to U_\infty \ \ \textrm{in}\  C^{2,1}_{loc}(\mathbb{R}^N\times \mathbb{R}),\end{equation} where $U_\infty$ is a solution of (\ref{eqEq}) such that
\begin{equation}\label{eqSandwich22}
   |U_\infty(x,t)|\leq \Phi (x),\ \ (x, t)\in \mathbb{R}^N\times \mathbb{R}.
\end{equation}
In the same manner, passing to a further subsequence if needed, we may assume that \begin{equation}\label{eqMasviV}V_n \to V_\infty \ \ \textrm{in}\ C^{2,1}_{loc}(\mathbb{R}^N\times \mathbb{R}),\end{equation} where $V_\infty$ is a solution of
\begin{equation}\label{eqVlast}
  V_t=\Delta V+p|U_\infty|^{p-1}V,\ \ x\in \mathbb{R}^N,\ t\in \mathbb{R},
\end{equation}
  such that
  \begin{equation}\label{eqQuot22}
  \lim_{|x|\to \infty}\frac{V_\infty(x,t)}{Z(x)}=0\  \textrm{holds uniformly in}\ t\in \mathbb{R}.
\end{equation}
Furthermore,  according to (\ref{eqRecal4}) (keeping in mind (\ref{eqX})), we have
\[
    V_\infty(x_\infty, 0) \geq \tilde{\lambda}Z(x_\infty) \ \ \textrm{and}\ \
          V_\infty(x, t) \leq \tilde{\lambda}  Z(x),\ \   (x,t)\in \mathbb{R}^N\times \mathbb{R}.
        \]
Arguing as in (\ref{eqZeq}), but using (\ref{eqSandwich22}) instead of (\ref{eqSandwich}), we can see that $Z$ is a super-solution of (\ref{eqVlast}) (alternatively, one can replace $t$ with $t+t_n$ in (\ref{eqZeq}) and pass to the limit).
As before, we can arrive at a contradiction which proves the theorem under the restriction that $\beta<\infty$.

We now turn to the case where $\beta=\infty$. This can be treated by choosing a 'stronger' positive supersolution, coming from (\ref{eqKernelSing}), and carefully modifying the above arguments near $x=0$. We will outline the necessary changes below.

Let us first note that, due to (\ref{eqSandwich}) (with $\beta=\infty$), it follows in a standard way that the  H\"{o}lder boundedness of $u$ and $v$ in (\ref{equB}) and (\ref{eqvB}), respectively, still holds provided that $|x|\geq \delta$ (for any $\delta>0$). These bounds, however, may degenerate as $\delta\to 0$ because the function $u(0,\cdot)$ could be unbounded.
The task is now to find an estimate for $v$ near $x=0$ which echoes  (\ref{eqSandwich}) for $u$. By the same method  as in Lemma \ref{lemExp}, it is easy to check that
\[
\Delta u=O(|x|^{-m-2}),\  \ \textrm{uniformly in}\ t\in \mathbb{R}, \ \textrm{as}\ x\to 0.
\]
Hence, via (\ref{eqEq}), we obtain

\[v=O(|x|^{-m-2}),\ \ \textrm{uniformly in}\ t\in \mathbb{R},\ \textrm{as}\ x\to 0.\]
Let us note in passing that the above bound is actually valid for all $x\neq 0$. Be that as it may, owing to (\ref{eqLambda12}), it is worse than (\ref{eqUpper}) as $|x|\to \infty$.

We are now in position to sweep with\begin{equation}\label{eqZinfty1234}Z_\infty:=\frac{1}{|x|^{m+\lambda_1}} ,\ \ x\neq 0,\end{equation}which is in the kernel of the linearization of (\ref{eqSS}) on $\varphi_\infty$ (recall (\ref{eqKernelSing})) and furnishes a supersolution to (\ref{eqv}) (again by (\ref{eqSandwich})). We emphasize that, in light of (\ref{eqLambda12}), $Z_\infty$ blows up at the origin faster than  $|v|$ can grow there as $t\to \pm \infty$ according to (\ref{eqSandwich}). More precisely, we have\begin{equation}\label{eqDiaz}
\lim_{x\to 0}\frac{v}{Z_\infty}=0\ \ \textrm{uniformly in}\ t\in\mathbb{R}.
\end{equation}

Naturally, in the definition of the set $\Lambda$ we have to require that the corresponding inequality (with $Z_\infty$ in place of $Z$) holds in $\left(\mathbb{R}^N\setminus \{0\}\right)\times \mathbb{R}$. This of course leads to some minor changes in regards to (\ref{eqLambdaBar}) and (\ref{eqLambdaTilt}). The point is that (\ref{eqDiaz}) forces $x_\infty\neq 0$ in the corresponding relation to (\ref{eqX}).

The only essential differences from this point on occur when dealing with the scenario
 $t_n\to -\infty$. We point out that the corresponding convergence to that in (\ref{eqMasviU}) holds over compacts that do not intersect the origin $\{x=0\}$.  In turn, we get the corresponding convergence property to (\ref{eqMasviV}) again over compacts away from the origin,  which provides  a  solution to (\ref{eqVlast}) for $x\neq 0$. Since $x_\infty\neq 0$ as we have remarked above, we can still apply the strong maximum principle around the point $(x_\infty,0)$ and arrive at a contradiction as before.
\end{proof} \begin{rem}\label{remLambdas}
From (1.13) in \cite{py} we have
\begin{equation}\label{eqlambda1}\begin{array}{c}
                                   \lambda_1=\frac{N-2-2m-\sqrt{(N-2-2m)^2-8(N-2-m)}}{2}, \\
                                     \\
                                   \lambda_2=\frac{N-2-2m+\sqrt{(N-2-2m)^2-8(N-2-m)}}{2}.
                                 \end{array}
\end{equation}

We first note that \[N-2m-2>4\ \ \textrm{ammounts to}\ \ \frac{4}{p-1}<N-6,\]
which holds  for $p\geq p_c$ on account of (\ref{eqJL}) (see also (4.70) in \cite{kelei}).
The inequality (\ref{eqLambda12}) then follows readily by observing that the quantity under the square root in the formula for
$\lambda_1$ can be conveniently bounded from above by
\[(N-2-2m)^2-8(N-2-2m)+16=(N-2m-6)^2.\]\end{rem}\begin{rem}\label{remGibbons}
As we have already mentioned, our proof of Theorem \ref{thmMine} is partly motivated by some recent rigidity results for entire solutions to a class of elliptic PDEs, mostly related to the Gibbons conjecture for the Allen-Cahn equation (see for instance
\cite{BH,weiChan}). We cannot resist the temptation to present a rigidity result of this kind for the steady state problem of (\ref{eqEq}).

For $N\geq 11$ and $M\geq 1$, we consider the elliptic equation
\[\Delta u+|u|^{p-1}u=0, \ z\in \mathbb{R}^{N+M}, \ \textrm{with}\ p\geq p_c(N),\]
and write $z=(x,y)$, $x\in \mathbb{R}^N$ and $y\in \mathbb{R}^M$.
Our basic assumption is that  a (classical) solution $u$ satisfies the bound (\ref{eqSandwich}) with $y$ in place of $t$ (with the obvious meaning).
Moreover, we assume that
\[
u(x,y)=\frac{L}{|x|^m}+o\left(\frac{1}{|x|^{m+\lambda_1-1}}\right), \ \textrm{uniformly in}\ y\in \mathbb{R}^M,\ \textrm{as}\ |x|\to \infty.
\]
We note in passing that these last two  assumptions are satisfied  if $u$ lies between two ordered positive radial solutions of the reduced problem (\ref{eqSS}) (in direct analogy to the assumption (\ref{eqSandwich9}) of \cite{py}).
Our assertion is that $u$ depends only on $x$.

The proof is a simple adaptation of that of Theorem \ref{thmMine}. Firstly, in the same manner as in the first part   of the proof of Lemma \ref{lemExp} we have
\[
|\nabla_y u|=o\left(\frac{1}{|x|^{m+\lambda_1}}\right), \ \textrm{uniformly in}\ y\in \mathbb{R}^M,\ \textrm{as}\ |x|\to \infty.
\]
Then, we can conclude that every component of the above gradient is identically equal to zero by  sweeping with $Z_\infty(x)$ as in the proof of Theorem \ref{thmMine}.
We point out that we could  have assumed an asymptotic relation that mirrors (\ref{eqlower222}) instead.
\end{rem}

\subsection{Proof of Theorem \ref{thmMine2}}\label{subsec3}\begin{proof}
We will only present the proof for $p>p_c$; the case of equality can be treated similarly, so we will just point out the main difference at the end of the proof (see also Remark \ref{remITE} below).

As in \cite[Prop. 3.4]{ni}, to establish that $u$ is radial we will show that $T_{ij}u\equiv 0$, $i,j=1,\cdots,N$, $i\neq j$, where
\[
T_{ij}=x_i\frac{\partial}{\partial x_j}-x_j\frac{\partial}{\partial x_i},\ \ x=(x_1,\cdots,x_N),
\]
denotes the tangential gradient along the circle of radius $|x|$ and center at the origin on the  $x_ix_j$ hyperplane.

Applying $T_{ij}$ to the equation for $u$, since the Laplacian is rotationally invariant, we see that
\begin{equation}\label{eqSweep22}
\Delta (T_{ij}u)+p|u|^{p-1}T_{ij}u=0,\ \ x\in \mathbb{R}^N.
\end{equation}

In the same vein as in the previous sections, we proceed to derive an asymptotic bound for $T_{ij}u$ as $|x|\to \infty$. To this end, recalling the formula for the singular solution $\varphi_\infty$ from (\ref{eqSingular}), we can write
\begin{equation}\label{eqGive}
  u=\varphi_\infty+\frac{a}{r^{m+\lambda_1}}+\psi,\ \ x\neq0,
\end{equation}
with
\[\psi=o\left(\frac{1}{|x|^{m+\lambda_1}}\right)\ \ \textrm{as}\ |x|\to \infty.\]
Substituting this in the equation for $u$, and keeping in mind the discussion leading to (\ref{eqKernelSing}), we obtain
\[
\begin{array}{rcl}
  -\Delta \psi & = & -\varphi_\infty^p-p\varphi_\infty^{p-1}\frac{a}{r^{m+\lambda_1}}+\left(\varphi_\infty+\frac{a}{r^{m+\lambda_1}}\right)^p \\
    &   &   \\
    &   & +\left|\varphi_\infty+\frac{a}{r^{m+\lambda_1}}+\psi\right|^{p-1}\left(\varphi_\infty+\frac{a}{r^{m+\lambda_1}}+\psi\right)-\left(\varphi_\infty+\frac{a}{r^{m+\lambda_1}}\right)^p \\
    &   &   \\
    & = & \frac{O(1)}{r^{(p-2)m+2m+2\lambda_1}}+\frac{o(1)}{r^{2+m+\lambda_1}} \\
    &   &   \\
    & = & \frac{O(1)}{r^{2+m+2\lambda_1}}+\frac{o(1)}{r^{2+m+\lambda_1}} \\
    &   &   \\
    & = & \frac{o(1)}{r^{2+m+\lambda_1}}
\end{array}
\]
as $|x|\to \infty$.
By working as in the first part of the proof of Lemma \ref{lemExp}, we deduce that
\[|\nabla\psi|=o\left(\frac{1}{|x|^{m+\lambda_1+1}}\right)\ \ \textrm{as}\ |x|\to \infty.\]
Consequently, going back to $u$ via (\ref{eqGive}), we get
\[
\frac{\partial u}{\partial x_j} =-\frac{x_j}{|x|}\frac{Lm}{|x|^{m+1}}-\frac{x_j}{|x|}\frac{a(m+\lambda_1)}{|x|^{m+\lambda_1+1}}+o\left(\frac{1}{|x|^{m+\lambda_1+1}} \right)
\ \ \textrm{as}\ |x|\to \infty,
\]
$j=1,\cdots,N$.
Hence, we infer that
\[
T_{ij}u=o\left(\frac{1}{|x|^{m+\lambda_1}} \right)
\ \ \textrm{as}\ |x|\to \infty,\ \ i,j=1,\cdots,N,\ i\neq j.
\]

Armed with the above information, we can sweep in the equation (\ref{eqSweep22}) for $T_{ij}u$ with the positive super-solution $Z_\infty$ (recall the definition (\ref{eqZinfty1234})) as in the proof of Theorem \ref{thmMine} to find  that $T_{ij}u\leq 0$ for any $i\neq j$. Doing the same for $-T_{ij}u$, we conclude  that $T_{ij}u$ is identically equal to zero for any $i\neq j$, which completes the proof of the theorem for $p>p_c$.

In the case $p=p_c$, the main difference is that now we have to sweep with the positive  weak supersolution
\[
\varphi_\infty-|u|
\]
of the linearized equation \footnote{This supersolution was brought to my attention by L. Dupaigne in relation to a point in \cite{sourggrok} (see also \cite[Prop. 1.3.2]{dupaigne2020}).} (see   \cite[Lem. I.1]{berestykia} for supersolutions whose gradients may have jump discontinuities). We note that we resort to this supersolution because  the analog of $Z_\infty$ for this purpose is sign changing due to the presence of the logarithm (recall (\ref{eqKernelSing})). As the reader may have already noticed, we needed to assume that $b<0$ so that the above supersolution has a good asymptotic behaviour as $|x|\to \infty$. On the other hand, if  $\beta<\infty$ in the bound (\ref{eqSandwich}) then it is easy to see that the proof goes through with  the usual positive supersolution $\varphi_\beta+\frac{1}{m}x\cdot \nabla \varphi_\beta$ without such an (a priori) information on $b\leq 0$.
\end{proof}

\begin{rem}\label{remITE}
If $p=p_c$, a careful examination of the above proof reveals that the assertion of the theorem
continues to hold under the weaker assumption that
\[u=\frac{L}{|x|^m}+b\frac{\ln |x|}{|x|^{m+\lambda_1}}+o\left(\frac{\ln |x|}{|x|^{m+\lambda_1}}\right)\ \ \textrm{as}\ |x|\to \infty,\]
for some $b<0$.
\end{rem}

\begin{rem}
If $p>p_c$, then Theorem \ref{thmMine2} follows  from Theorem 1.1 of \cite{PYA} concerning the asymptotic behaviour of the Cauchy problem (\ref{eqCauchy}).
Nevertheless, our proof carries over with just slight modifications to the case where $u$ is a solution in $\mathbb{R}^N\setminus \{0\}$ (with $p\geq p_c$).
Actually, it is a simple matter to check that $T_{ij}u=O(|x|^{-m})\ll Z_\infty$ as $x\to 0$ which is more than what is needed for the last part of the proof of Theorem \ref{thmMine} to carry over to this setting.
\end{rem}

\begin{rem}\label{remWang}
An important observation for the steady states of (\ref{eqEq}) that are bounded in absolute value by the singular solution $\varphi_\infty$ is that
\begin{equation}\label{eqlimes}
  \lim_{|x|\to \infty}|x|^{\frac{2}{p-1}}u(x)\  \textrm{exists and is equal to}\ -L\ \textrm{or}\ 0 \ \textrm{or}\ L,
   \end{equation}where $L>0$ is as in (\ref{eqSingular}).
This can be proven in the spirit of  the classical regularity theory of minimal
surfaces: using a suitable   monotonicity formula, a blow up analysis, and   classifying
the limiting (conic) solutions that are equivariant by dilation. We give only the main ideas of the proof.

For $R\gg 1$, we consider the rescalings
\begin{equation}\label{eqSkalinas}
w_R(x)=R^{\frac{2}{p-1}}u(Rx),\ \ x\in \mathbb{R}^N.
\end{equation}
Clearly, $w_R$ is still a steady state of (\ref{eqEq}) and  bounded in absolute value by $\varphi_\infty$.
Therefore, by standard elliptic estimates and the usual diagonal argument, we can find a diverging sequence $\{R_n\}$ such that
\begin{equation}\label{eqWangUoy1}
w_{R_n}\to w_\infty\  \textrm{in}\ C^2_{loc}\left(\mathbb{R}^N\setminus \{0\} \right)\  \textrm{as}\ n\to \infty,
\end{equation}
where $w_\infty$ is a steady state for $x\neq 0$ and bounded in absolute value by $\varphi_\infty$.

Using the monotonicity formula of \cite{pacard}, which is invariant under the  scaling in (\ref{eqSkalinas}),  as in the case of minimal surfaces or free boundary problems we deduce   that
the 'blow up' limit $w_\infty$ is homogeneous, i.e.,
\[
w_\infty(\lambda x)=\lambda^{-\frac{2}{p-1}}w_\infty(x),\ x\in \mathbb{R}^N\setminus \{0\}, \ \forall \ \lambda>0,
\]
(see also \cite{kelei}).

By the above property and the fact that $w_\infty \in C^\infty\left(\mathbb{R}^N\setminus \{0\} \right)$, it follows that there exists an $f\in C^\infty\left(\mathbb{S}^{N-1} \right)$
($\mathbb{S}^{N-1}$ stands for the unit sphere in $\mathbb{R}^N$)
such that in polar coordinates we have
\[
w_\infty(r,\theta)=r^{-\frac{2}{p-1}}f(\theta),\ \ r>0,\ \theta\in \mathbb{S}^{N-1}.
\]
It is a simple matter to check  that $f$ satisfies
\begin{equation}\label{eqBeltra2}
  \Delta_{\mathbb{S}^{N-1}}f-L^{p-1}f+|f|^{p-1}f=0\ \ \textrm{on}\ \mathbb{S}^{N-1},
\end{equation}
where $\Delta_{\mathbb{S}^{N-1}}$ denotes the Laplace-Beltrami operator on $\mathbb{S}^{N-1}$ (with the standard metric).
On the other hand, since  $w_\infty$ is bounded in absolute value by $\varphi_\infty$ it follows that
\begin{equation}\label{eqfb}
\left|f(\theta) \right|\leq L,\ \ \theta \in \mathbb{S}^{N-1}.
\end{equation}
Testing the equation (\ref{eqBeltra2}) by $f$ and integrating by parts yields
\[
\int_{\mathbb{S}^{N-1}}^{}\left\{|\nabla_{\mathbb{S}^{N-1}}f|^2+f^2\left(L^{p-1}-|f|^{p-1} \right)\right\}dS(x)=0.
\]
Consequently, according to (\ref{eqfb}) we get $f\equiv -L$ or $f\equiv 0$ or $f\equiv L$. Since the blow up limits are finitely many, it is easy to see that there has to be a unique one. So,  the convergence in (\ref{eqWangUoy1}) takes place for any diverging sequence $\{R_n\}$. By the definition of the blow up $w_R$, we conclude that
(\ref{eqlimes}) holds.

If $u$ is sign definite, then this asymptotic behaviour forces it to be  radially  symmetric
for $p$ in certain regimes.
Indeed, assuming that $u$ is positive (without loss of generality) and that  the  last case in (\ref{eqlimes}) is true,  this property can be deduced from the discussion  in Subsection \ref{subsubRadialp1}   at least when  $p\in [p_c,\infty)\setminus [(N+1)/(N-3),N/(N-4))$. On the other hand, if the limit in (\ref{eqlimes}) is the trivial one,
the radial symmetry property can still be shown  by the method of moving planes (for any $p>1$, see  \cite[Thm. 2.6]{pacella}).

If $u$ is sign changing,  the  aforementioned method cannot be applied. Nevertheless, we can assert that $u$ is identically equal to zero in the case of the trivial limit in (\ref{eqlimes}) as in \cite[pg. 553]{farina22}. % under the slightly more restrictive assumption that it is bounded in absolute value by some positive radial steady state $\varphi_\beta$ with $\beta \in (0,\infty)$. This can be proven by   sweeping from above  directly  in the equation for $u$ and $-u$  with the family of solutions $\varphi_\alpha$, $\alpha \in (0,\beta]$.
One may ask whether the radial symmetry of $u$ is still true in the other cases of (\ref{eqlimes}). This question is at present far from being solved.
\end{rem}

\textbf{Acknowledgments.} The author gratefully acknowledges that his motivation for the connection with mean curvature flow in Subsection \ref{subsecOpen} came from
an open discussion of Prof. Cabr\'{e} with Prof. del Pino during the course of the latter in the summer school JISD2018. He wishes to express his thanks to Prof. Pol\'{a}\v{c}ik for bringing to his attention the papers \cite{filaYana} and \cite{polacikQui}, following the first version of the current work. Lastly, the author is indebted to IACM of FORTH, where this paper was written, for the hospitality.  This work has received funding from the Hellenic Foundation for Research and Innovation (HFRI) and the General Secretariat for Research and Technology (GSRT), under grant agreement No 1889.

\end{document}